\documentclass[letterpaper, 10 pt, conference]{ieeeconf}  

\IEEEoverridecommandlockouts                              
\usepackage{epsfig} 
\usepackage{amsmath,amssymb,amsfonts}
\usepackage{algpseudocode}
\usepackage{algorithm}
\usepackage{graphicx}
\usepackage{graphics}
\usepackage{textcomp}
\usepackage{xcolor}
\usepackage{subcaption}
\usepackage{cleveref}
\usepackage[font=small,skip=0pt]{caption}

\newtheorem{theorem}{Theorem}
\newtheorem{lemma}{Lemma}

\newtheorem{proposition}{Proposition}    
\newtheorem{assumption}{Assumption}    
\newtheorem{definition}{Definition}

\usepackage{algorithm}
\usepackage{algpseudocode}

\newcommand{\magenta}[1]{{\color{black}#1}}

\newcommand{\revise}[1]{{\color{black}{#1}}}

\title{\LARGE \bf
Anytime-Feasible Gradient Descent for Constrained Optimization Under Gradient Uncertainty
}

\author{Sina Sharifi, Jiarui Wang, and Mahyar Fazlyab
\thanks{The authors are with the Department of Electrical and Computer Engineering at Johns Hopkins University, Baltimore, MD 21218, USA.
{\tt\small \{sshari12, jwang486, mahyarfazlyab\}@jhu.edu}.
This work was supported  by the
National Science Foundation (NSF) under Grant 2515978 and by the
DOE Office of Science (ASCR) under Award No. 826565.
}
}

\begin{document}

\maketitle
\thispagestyle{empty}
\pagestyle{empty}

\begin{abstract}
Constrained optimization is central to many engineering systems in which decisions must satisfy strict safety and operational requirements, especially in real-time settings with limited computational budgets. In such scenarios, optimization algorithms are often terminated before full convergence, making \emph{anytime feasibility} essential for safe deployment. Existing methods that guarantee feasibility at every iterate typically rely on exact gradient information, an assumption that is often violated in practice due to measurement noise, stochastic approximations, or model mismatch. We develop an anytime-feasible first-order method for nonlinear constrained optimization under norm-bounded errors in the objective and constraint gradients. The method computes a robust search direction by solving a second-order cone program and selects a step size through safeguarded backtracking. Assuming exact function evaluations and a strictly feasible
initialization, the method preserves strict feasibility and guarantees sufficient objective decrease whenever the computed search direction is nonzero. We establish a uniform positive lower bound on the accepted step sizes, an $O(1/K)$ bound on the average squared search direction norm, and convergence of the search directions to zero. We also show that a zero search direction at a strictly feasible point certifies approximate first-order stationarity. We validate the proposed method on a multi-agent navigation task in cluttered environments and show that it maintains collision-free trajectories despite noisy gradient information.
\end{abstract}

\section{INTRODUCTION}

Constrained optimization underlies many modern engineering systems, where decisions must be made while satisfying strict operational and safety requirements \cite{erseghe2014distributed, wensing2023optimization}. This is particularly critical in real-time and high-frequency settings, where solutions must be computed rapidly while remaining robust to uncertainty and model inaccuracies. A representative example arises in trajectory optimization for autonomous systems operating in cluttered environments~\cite{williams2016aggressive}, where a robot must compute collision-free trajectories while minimizing energy or travel time using approximate gradient information. Limited computational budgets often require terminating the optimization before convergence. Ensuring anytime feasibility, i.e., satisfying the constraints at every iteration, is therefore essential for safely deploying intermediate solutions.

Motivated by these challenges, a number of methods have been proposed to guarantee anytime feasibility. Most existing approaches, however, rely on access to exact gradients of the objective and constraint functions \cite{auslender2010moving, wang2025anytime}. This assumption is often unrealistic in practice, where gradients are corrupted by noise due to measurement errors, stochastic approximations, or model mismatch.
A complementary line of work studies constrained optimization in the presence of noisy or stochastic gradient information \cite{oztoprak2023constrained, berahas2025line, paquette2020stochastic}. However, these methods do not guarantee anytime feasibility.

In this work, we bridge these two directions by studying constrained optimization problems under noisy gradient information and developing a method that ensures the iterates remain feasible at all times. Specifically, we first formulate the problem as a sequence of quadratically constrained quadratic programs (QCQPs) and show that, when accounting for worst-case noise, each subproblem reduces to a second-order cone program (SOCP). We then demonstrate that solving this SOCP, combined with a line search procedure, yields safe update directions and guarantees convergence at an ergodic rate of $\mathcal{O}(1/K)$ in terms of the average squared search-direction norm.


\subsection{Related Work}
Sequential Quadratic Programming (SQP) is a widely used class of methods for solving nonconvex constrained optimization problems. These methods iteratively approximate the original nonlinear program by solving a sequence of quadratic subproblems, obtained via a second-order expansion of the Lagrangian and linearization of the constraints \cite{gill1986some, mayne2009surperlinearly}. 
Standard SQP iterations do not, in general, preserve feasibility of the original nonlinear constraints, hindering their deployment in online settings where intermediate iterates must remain feasible.
%
Assuming knowledge of the Lipschitz constants of the gradients, \cite{auslender2010moving} constructs conservative local approximations, thereby ensuring anytime feasibility.
%
In contrast, \cite{sharifi2025sequential, wang2025anytime} incorporate a tilting quadratic term into the linearized constraints, yielding a QCQP subproblem, and use a line search procedure to ensure anytime feasibility without requiring prior knowledge of the gradient Lipschitz constants.

The methods discussed so far all assume access to exact gradients when needed. However, a complementary line of work considers the optimization problem in the presence of noise. The authors of  \cite{oztoprak2023constrained} propose a noise-tolerant version of SQP and provide guarantees that the iterates converge to a neighborhood of the solution.
Other methods change the classical Armijo condition and adapt it to the stochastic setup \cite{berahas2025line, paquette2020stochastic}.
In the context of control synthesis for dynamical systems with uncertainty, \cite{long2022safe} uses control barrier functions (CBFs)~\cite{ames2016control} to enforce safety under uncertainty, leading to a reformulation as a second-order cone program (SOCP) that can be solved efficiently online.
In the context of reinforcement learning, \cite{mestres2025off} considers stochastic gradient uncertainty arising from finite-sample estimation and provides probabilistic safety guarantees. In contrast, we consider worst-case (deterministic) gradient perturbations and enforce feasibility for all realizations within a prescribed uncertainty set.

\section{Background and Problem Statement}

\subsection{Problem Statement}
We consider the inequality-constrained optimization problem
\begin{align}\label{eq:optimization_problem}
    f^\star =\min_{x}~ \{  f_0(x) 
    \ \ \text{s.t.}~ f_i(x) \leq 0, \quad i \in [m]\},
\end{align}
where \(f_i:\mathbb{R}^n \to \mathbb{R}\), \(i\in[0,m]:=\{0, \cdots, m\}\), are continuously differentiable.
For brevity, we denote $[1, m]$ by $[m]$.
We assume that the optimal value \(f^\star > -\infty\) is attained at some feasible point \(x^\star\). We denote the feasible set by
\[
\mathcal{F} = \{ x \in \mathbb{R}^n \mid f_i(x) \le 0,\ \forall i \in [m] \}.
\]
We assume that the objective and constraint values can be
evaluated exactly, while their gradients are available through
estimates satisfying
\begin{equation}\label{eq:noisy_grad}
    \nabla f_i(x)=\widehat{\nabla f}_i (x)+d_i,
    \quad \|d_i\|\le\epsilon_i,
    \quad i\in\{0,\ldots,m\},
\end{equation}
where $\epsilon_i>0$ are known error bounds.
The estimates may vary between iterations; dependence on
the current estimates is suppressed in the notation $u(x)$.
We make the following assumptions throughout the paper.
\begin{assumption}[Smoothness]\label{ass:smooth}
For each \(i \in [0,m]\), the function \(f_i\) is continuously differentiable and has \(L_i\)-Lipschitz continuous gradient, i.e.,
\[
\|\nabla f_i(x)-\nabla f_i(y)\| \le L_i \|x-y\|,
\qquad \forall x,y \in \mathbb{R}^n.
\]
\end{assumption}

\begin{assumption}[MFCQ]\label{ass:mfcq}
The Mangasarian--Fromovitz constraint qualification~\cite{nocedal2006numerical} holds at every feasible point \(x \in \mathcal F\); that is, for every \(x \in \mathcal F\), there exists a direction \(d \in \mathbb{R}^n\) such that
\(
\nabla f_i(x)^\top d < 0,
\quad \forall i \in \{i \mid f_i(x) = 0\}.
\)
\end{assumption}

\begin{definition}[$\varepsilon$-KKT point]\label{def:e-kkt}
A point $x \in \mathcal{F}$ is called an $\varepsilon$-KKT point if there exists $\lambda \in \mathbb{R}_+^m$ such that
\begin{align*}
    \big\| \nabla f_0(x) + \sum_{i=1}^m \lambda_i \nabla f_i(x) \big\| &\le \varepsilon, \\
    \lambda_i \geq 0, \ f_i(x) \leq 0, \
    \lambda_i f_i(x) &\geq -\varepsilon, \quad \forall i \in [m].
\end{align*}
\end{definition}

\subsection{Deterministic Safe Sequential QCQP}
To motivate our development, we first recall the deterministic safe gradient method proposed in \cite{wang2025anytime}. At a current feasible iterate \(x\), the search direction \(\tilde{u}(x)\) is obtained as
\begin{align}\label{eq:qcqp}
\tilde{u}(x) \! =\! &\arg\min_{u} \tfrac{1}{2} \|u + \nabla f_0(x)\|^2 \\
&\text{s.t.} ~  \nabla f_i(x)^\top u + \alpha f_i(x) + w_i \|u\|^2 \le 0, ~ \forall i \in [m]. \notag
\end{align}
where $\alpha>0$, $w_i>0$. The objective chooses a search direction that deviates minimally from the unconstrained gradient-descent direction, whereas the constraints enforce barrier-type invariance conditions associated with the feasible set. In particular, the inequality \(\nabla f_i(x)^\top \tilde{u}(x) + \alpha f_i(x)\le 0\) is the nominal continuous-time CBF condition for the constraint \(f_i(x)\le 0\), and the additional quadratic term \(w_i\|\tilde u(x)\|^2\) introduces a robustness margin that accounts for discretization error in the subsequent line-search-based implementation. Specifically, after computing \(\tilde{u}(x)\), the method in \cite{wang2025anytime} employs a backtracking line search to select a step size \(\eta\) and update the iterate as \(x^+ = x + \eta \tilde{u}(x)\). The line search is designed to guarantee two properties simultaneously: i) \emph{anytime feasibility}, meaning that every iterate remains in \(\mathcal F\), and ii) \emph{monotonic descent} of the objective. A key feature of \eqref{eq:qcqp} is the quadratic safeguard term \(w_i\|\tilde u(x)\|^2\), which ensures that the accepted step size is bounded away from zero under suitable smoothness assumptions. This avoids the vanishing-step-size pathology that can arise when discretizing continuous-time QP-based safe-gradient dynamics; see, e.g., \cite{muehlebach2021constraints,allibhoy2023control,sharifi2025safe}.

\begin{figure}[t]
    \centering
    \includegraphics[width=0.99 \columnwidth]{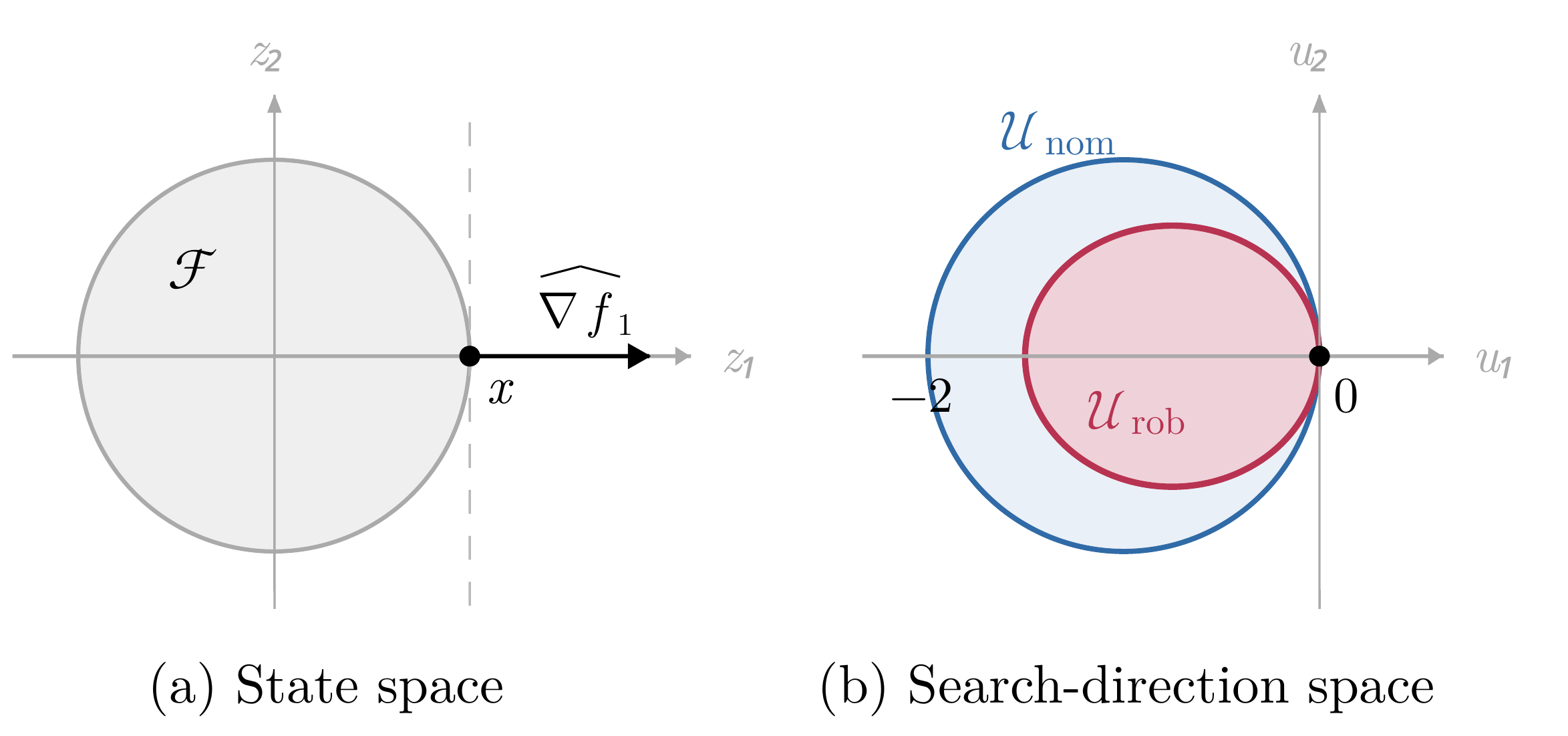}
    \caption{
    Geometry of the nominal and robust search-direction constraints.
(a) Feasible set $\mathcal F$, boundary point $x$, and estimated
gradient direction $\widehat {\nabla f}_1$.
(b) Gradient uncertainty contracts the nominal admissible set
$\mathcal U_{\mathrm{nom}}$ to $\mathcal U_{\mathrm{rob}}$,
excluding directions insufficiently aligned with the inward normal.
    }
    \label{fig:overview}
\end{figure}

\section{Sequential SOCP}

\subsection{Robust Search Direction via SOCP}
We now revisit the deterministic QCQP \eqref{eq:qcqp} with uncertain gradients. Recall that, instead of the exact gradients, we only have access to inexact estimates satisfying \eqref{eq:noisy_grad}. To preserve descent and anytime feasibility under gradient uncertainty, we enforce each constraint robustly over all admissible perturbations.
Specifically, for each \(i \in [m]\), we replace the nominal constraint in \eqref{eq:qcqp} by
\begin{align}
    \sup_{\|d_i\|\le \epsilon_i}
    \Big[(\widehat{\nabla f}_i(x)+d_i)^\top u\Big]
    + \alpha f_i(x) + w_i\|u\|^2 \le 0.
    \label{eq:robust_constraint_sup}
\end{align}
We have $\sup_{\|d_i\|\le \epsilon_i} d_i^\top u = \epsilon_i \|u\|$, 
with maximizer \(d_i^\star = \epsilon_i u/\|u\|\) whenever \(u\neq 0\). Hence, \eqref{eq:robust_constraint_sup} is equivalent to
\begin{align}\label{eq:robust_const}
    \widehat{\nabla f}_i(x)^\top u \!+\! \epsilon_i\|u\| \!+\! \alpha f_i(x) \!+\! w_i\|u\|^2 \le 0,
     \forall i \in [m].
\end{align}
We now turn to the objective. In the deterministic QCQP \eqref{eq:qcqp}, the objective can be written as
\[
\frac{1}{2}\|u+\nabla f_0(x)\|^2
=
\frac{1}{2}\|u\|^2+\nabla f_0(x)^\top u+\frac{1}{2}\|\nabla f_0(x)\|^2.
\]
Since \(x\) is fixed when solving the subproblem, the last term is independent of \(u\) and can be omitted without affecting the minimizer. Motivated by this equivalent reduced form, we robustify the reduced objective model $\frac{1}{2}\|u\|^2+\nabla f_0(x)^\top u$ under the uncertainty model \eqref{eq:noisy_grad}. This leads to the worst-case objective
\begin{align}\label{eq:robust_obj}
    \sup_{\|d_0\|\le \epsilon_0}
    & \left[
        \frac{1}{2}\|u\|^2 + u^\top 
        (\magenta{\widehat{ \nabla f}}_0(x) + d_0)
    \right] \notag \\  
    &\quad =\frac{1}{2}\|u\|^2 + \magenta{\widehat{ \nabla f}}_0(x)^\top u + \epsilon_0\|u\|.
\end{align}
Combining \eqref{eq:robust_const} and \eqref{eq:robust_obj}, we obtain the robust search direction \(u(x)\) as the solution of
\begin{align}\tag{SS-SOCP}
    &u(x) \! = \! \arg\!\min\ 
    \frac{1}{2}\|u\|^2 \! + \! \magenta{\widehat{ \nabla f}}_0(x)^\top u + \epsilon_0\|u\|
    \label{eq:robust_qcqp}\\
    &\mathrm{s.t.}
    \widehat{\nabla f}_i(x)^\top u \!+ \!\epsilon_i\|u\| \!+ \!\alpha f_i(x) \!+\! w_i\|u\|^2 \le 0, \forall i \! \in \! [m].
    \notag
\end{align}
This problem is a robust counterpart of \eqref{eq:qcqp}, where both descent and the forward invariance constraints are enforced against worst-case gradient perturbations.
\Cref{fig:overview} visualizes the differences between the feasible set of \eqref{eq:qcqp} and \eqref{eq:robust_qcqp}.
The resulting problem \eqref{eq:robust_qcqp} 
%
admits an equivalent SOCP representation as follows,
\begin{align} 
\min_{u,r,s} \quad 
& \tfrac{1}{2}s + \magenta{\widehat{ \nabla f}}_0(x)^\top u + \epsilon_0 r \\
\text{s.t.} \quad 
& \widehat{\nabla f}_i(x)^\top u + \epsilon_i r + \alpha f_i(x) + w_i s \leq 0, 
\quad \forall i \in [m], \notag \\
& \|u\| \le r, \quad \|u\|^2 \le s, \notag
\end{align}
An optimal solution can be chosen with \(r \!=\!\|u\|\) and \(s \! =\! r^2 \!= \! \|u\|^2\), so this formulation is equivalent to \eqref{eq:robust_qcqp}. Moreover, \(\|u\|\le r\) is a standard second-order cone constraint, while \(\|u\|^2 \le s\) admits a rotated second-order cone representation. Hence, the problem can be efficiently solved using off-the-shelf SOCP solvers.

In the following result, we show that \eqref{eq:robust_qcqp} yields a descent direction for the objective despite noisy gradient information.
To this end, we first establish strong duality for \eqref{eq:robust_qcqp}.

\smallskip
\begin{lemma}[Slater's Condition]\label{lem:strict_feas}
Suppose \(x\) is strictly feasible for \eqref{eq:optimization_problem}, i.e., $f_i(x) < 0, \ \ \forall i \in [m]$. Then the robust subproblem \eqref{eq:robust_qcqp} is strictly feasible. Therefore, Slater's condition holds for \eqref{eq:robust_qcqp}.
\end{lemma}
\begin{proof}
Setting $u = 0$ yields $\alpha f_i(x) < 0$ for all $i \in [m]$, hence all constraints are strictly satisfied.
\end{proof}

\smallskip
\begin{theorem}[Robustness of Descent Direction]\label{thm:descent}
    Let \(x\) be a strictly feasible point for \eqref{eq:optimization_problem}, and let \(u(x)\) be the optimizer of \eqref{eq:robust_qcqp}. Then
\(
\nabla f_0(x)^\top u(x) \leq -\|u(x)\|^2.
\)    
\end{theorem}
\begin{proof}
    The proof can be found in \Cref{proof:descent}
\end{proof}



\subsection{Sequential SOCP with Safeguarded Backtracking}

We now use the search direction obtained by solving \eqref{eq:robust_qcqp} to generate the iterates
\begin{align}\label{eq:update}
x^{k+1} = x^k + t^k u(x^k),
\end{align}
where $t^k>0$ denotes the step size. To select $t^k$ adaptively, we employ a safeguarded backtracking line search.  
Specifically, given fixed parameters $\gamma_1, \gamma_2\in(0,1)$, we choose $t^k$ such that
\begin{align}
f_0(x^k + t^k u(x^k)) &\le f_0(x^k) - \gamma_1 t^k \|u(x^k)\|^2, \label{eq:armijo} \\
f_i(x^k + t^k u(x^k)) &\le (1 - \gamma_2) f_i(x^k), \qquad \forall i \in [m]. \label{eq:feas_ls}
\end{align}
Condition \eqref{eq:armijo} imposes an Armijo-type sufficient decrease on the objective, while \eqref{eq:feas_ls} permits the constraint to approach its boundary while remaining strictly satisfied. In \Cref{prop:stepsize} we prove the existence of such a step size $t^k > 0$.

The next result shows that if the subproblem \eqref{eq:robust_qcqp} returns the zero direction at a strictly feasible point, then that point satisfies the $\varepsilon$-KKT conditions of the original problem. Hence, the condition $u(x)\!=\!0$ serves as a natural first-order stationarity certificate for the proposed method.

\smallskip
\begin{theorem}[$\varepsilon$-KKT conditions]
\label{thm:e-kkt}
Assume that $x$ is strictly feasible for \eqref{eq:optimization_problem}, i.e., $f_i(x)<0, \quad \forall i\in[m]$. 
Let $u(x)$ be an optimal solution of \eqref{eq:robust_qcqp}. If $u(x)=0$, then $x$ is an $\varepsilon$-KKT point of \eqref{eq:optimization_problem} with $\varepsilon = 2\epsilon_0$. 
\end{theorem}
\begin{proof}
    The proof can be found in  \Cref{proof:e-kkt}.
\end{proof}

\smallskip
\begin{proposition}[Lower bound on step size]\label{prop:stepsize}
Let \Cref{ass:smooth} hold, and suppose $x^k$ is strictly feasible for \eqref{eq:optimization_problem}, then condition \eqref{eq:armijo} and \eqref{eq:feas_ls} hold for all $t \in [0, t_{\min}]$ where
\begin{align}
t_{\min} \! = \beta \! \min\left\{ \frac{2(1-\gamma_1)}{L_0}, \; \min_i\left(\frac{2w_i}{L_{i}}\right), \; \frac{\gamma_2}{\alpha}, \; t_{\mathrm{init}} \right\}, \notag
\end{align}
and $t_{\mathrm{init}}$ is the initial step size in \Cref{alg:method}.
Consequently, the safeguarded backtracking line search returns $t^k \ge t_{\min}$.
\end{proposition}
\begin{proof}
    The proof can be found in \Cref{proof:stepsize}.
\end{proof}


\begin{theorem}[Strict feasibility and convergence]
\label{thm:convergence}
Suppose Assumption~1 and the gradient uncertainty model~(2)
hold. Let $\{x^k\}_{k\ge0}$ be generated by Algorithm~1  from a
strictly feasible point $x^0$. Assume
$\alpha>0$, $w_i>0$ for all $i\in[m]$,
$\beta,\gamma_1,\gamma_2\in(0,1)$. Then every iterate is strictly feasible, and
\begin{equation}\label{eq:strict_feasibility_bound}
    f_i(x^K)
    \le (1-\gamma_2)^K f_i(x^0)<0,
    \qquad i\in[m],\ K\ge0.
\end{equation}
Moreover, for every $K\ge1$,
\begin{equation}\label{eq:ergodic_direction_bound}
\begin{aligned}
    \frac{1}{K}\sum_{k=0}^{K-1}\|u(x^k)\|^2
    &\le
    \frac{f_0(x^0)-f^\star}
         {\gamma_1t_{\min}K}.
\end{aligned}
\end{equation}
where \(t_{\min}\) is defined in Proposition 1. In particular,
$\sum_{k=0}^{\infty}\|u(x^k)\|^2
    \le
    \frac{f_0(x^0)-f^\star}{\gamma_1t_{\min}}
    <\infty,$
and hence
\begin{equation}\label{eq:direction_convergence}
    \lim_{k\to\infty}\|u(x^k)\|=0.
\end{equation}
\end{theorem}



\begin{algorithm}[t]
\caption{Safe Sequential SOCP with Line Search}
\label{alg:method}
\begin{algorithmic}[1]
    \Require Initialization $x^0 \in \mathbb{R}^{n}$, maximum iteration $K > 0$, stopping tolerance $\epsilon_{tol}>0$, line search parameters $\gamma_1, \gamma_2,\beta \in (0,1)$, $\alpha>0$, $w_i>0$, gradient-error bound \(\epsilon_i>0\), and initial step size $t_{\mathrm{init}} > 0$
    
    \State Ensure $f_i(x^0) < 0, \ \ i \in [m]$

    \For{\revise{$k = 0,1,2,...,K-1$}}
        \State Compute $u(x^k)$ according to \eqref{eq:robust_qcqp}
        \revise{
        \If {$\|u(x^k)\|_2 \leq \epsilon_{tol}$}
            \Return $x^k$
        \EndIf
        }
        \State $t \gets t_{\mathrm{init}}$
        
        \While{either \eqref{eq:armijo} or \eqref{eq:feas_ls} fails}
            \State $t \gets \beta t$
        \EndWhile
        
        \State $t^k \gets t, \quad x^{k+1} \gets x^k + t^k u(x^k)$ 
    \EndFor
    
    \State \textbf{Return} $x^{K}$
\end{algorithmic}
\end{algorithm}

\section{Experiments}
In this section, we evaluate the proposed method on a multi-agent optimal control task involving collision-free navigation. The objective is to steer $n_v$ vehicles from their initial configurations to designated goal states while avoiding both static obstacles and inter-agent collisions. The experimental setup follows \cite{wu2025accelerated, wang2025anytime}. Formally, the problem is
\begin{align} \label{eq:MPC_experiment}
\min_{X, U} ~ & \sum_{t=0}^{N-1} l(X(t), U(t)) + V(X(N)) \\
\text{s.t.} ~
& U_i(t) \in \mathcal{U}, ~ \forall t \in [0, N-1], \ i \in [n_v], \notag \\
& X_i(t) \in \mathcal{X}, ~ \forall t \in [0, N], \ i \in [n_v], \notag \\
& X_i(t+1) \! = \! f(X_i(t), U_i(t)), ~ \forall t \in [0, N-1], i \in [n_v], \notag \\
& X_i(0) = X_i^s, ~\forall i \in [n_v], \notag \\
& g_j(X_i(t)) \! \leq \! 0, ~\forall i \in [n_v], j \in [3], t \in [0, N], \notag \\
&\|p_i(t)-p_j(t)\|_2^2 \ge 1,
\quad 1\le i<j\le n_v,\quad t\in [0, N]. \notag
\end{align}
Each vehicle follows Dubins' car dynamics:
\magenta{
\begin{align*}
f(X_i(t), U_i(t)) =
\begin{bmatrix}
    x_i(t) + v_i(t)\Delta t\cos\theta_i(t) - \Delta x \Delta t \\
    y_i(t) + v_i(t)\Delta t\sin\theta_i(t) \\
    \theta_i(t) + w_i(t)\Delta t
\end{bmatrix}.
\end{align*}
where $\Delta t$ is the discretization step size, $X_i(t) = [x_i(t), y_i(t), \theta_i(t)]^\top$ consists of $xy$ coordinate and orientation, and $U_i(t) = [v_i(t), w_i(t)]^\top$ consists linear and angular velocity. The constant $\Delta x=1$ represents
a drift velocity in the negative $x$-direction. We write $p_i(t)=[x_i(t),y_i(t)]^\top$.}
%
%
\magenta{
The desired steady-state input is $U_i^d = [1,\,0]^\top$ for all $i \in [n_v]$. The stage cost is
\[
    l(X(t), U(t)) \! = \!
    \sum_{i \in [n_v]} 
        \big(
            \|X_i(t) - X_i^d\|_2^2 
            + 0.01\,\|U_i(t) - U_i^d\|_2^2
        \big),
\]
and the terminal cost is
\(
    V(X(N)) 
    = \sum_{i \in [n_v]} 
        \|X_i(N) - X_i^d\|_{P_i}^2,
\)
where each $P_i$ is obtained from the discrete-time algebraic Riccati Equation 
\begin{small}
\begin{align*}
    A_i^\top P_i A_i - P_i - (A_i^\top P_i B_i)(R + B_i^\top P_i B_i)^{-1}(B_i^\top P_i A_i) + Q \! = \! 0,
\end{align*}
\end{small}
with $A_i = \nabla_1 f(X_i^d, U_i^d)$, $B_i = \nabla_2 f(X_i^d, U_i^d)$, $Q=I$ and $R=0.01 I$. The state and input constraints are defined as
\[
\begin{aligned}
\mathcal{X} &= 
\{(x, y, \theta) : |x| \le 3.2, |y| \le 3.2,~ |\theta| \le \pi\},\\
\mathcal{U} &= 
\{(v, w) : -5 \le v \le 12,~  |w| \le \tfrac{3}{2}\pi\}.
\end{aligned}
\]
Obstacle avoidance is enforced through the constraint
\(
    g_j(X_i(t)) = r_j^2 - \|p_i(t) - c_j\|_2^2 \le 0,
\)
where obstacle centers are $c_1 = [-1, -1]^\top, c_2 = [0,1]^\top, c_3 = [1,0]^\top$ and obstacles radii are $r_1=1, r_2=0.5, r_3=0.5$.
We set the planning horizon to $N = 30$ and $\Delta t = 0.03$.
}

\begin{figure}[t]
    \centering
    \includegraphics[width=0.49\linewidth]{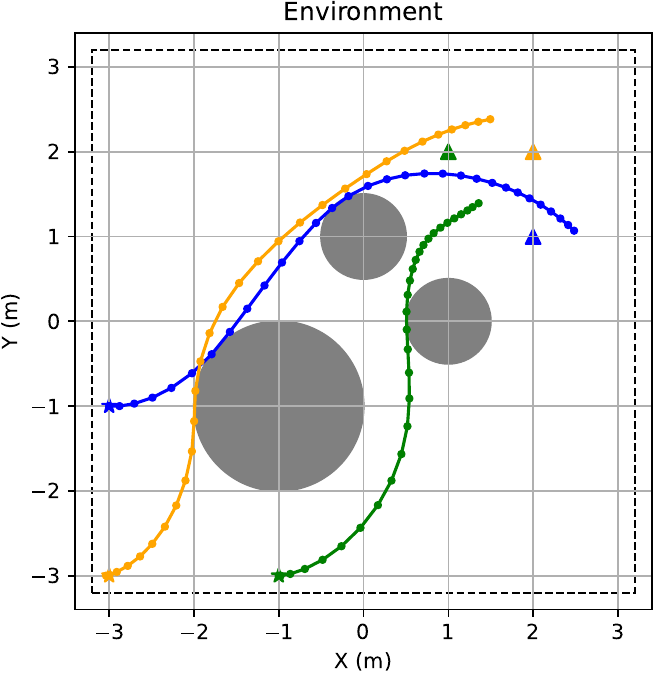}
    \includegraphics[width=0.49\linewidth]{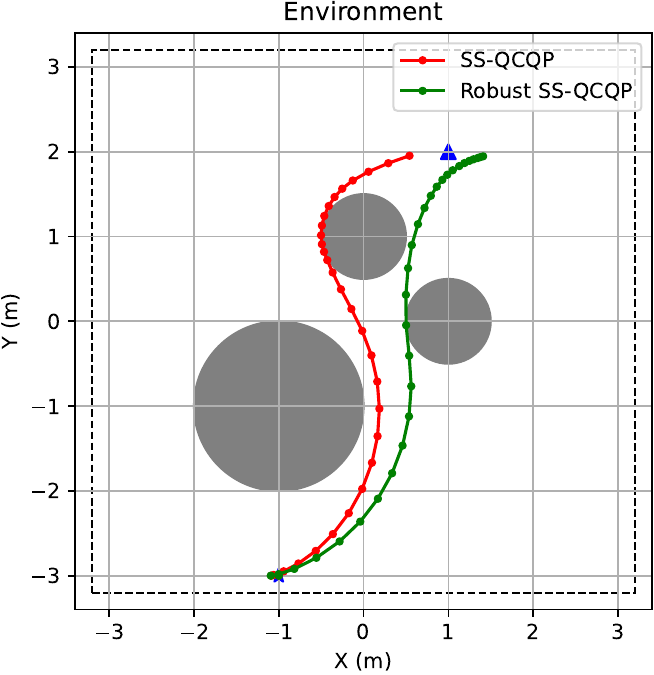}

    \includegraphics[width=0.49\linewidth]{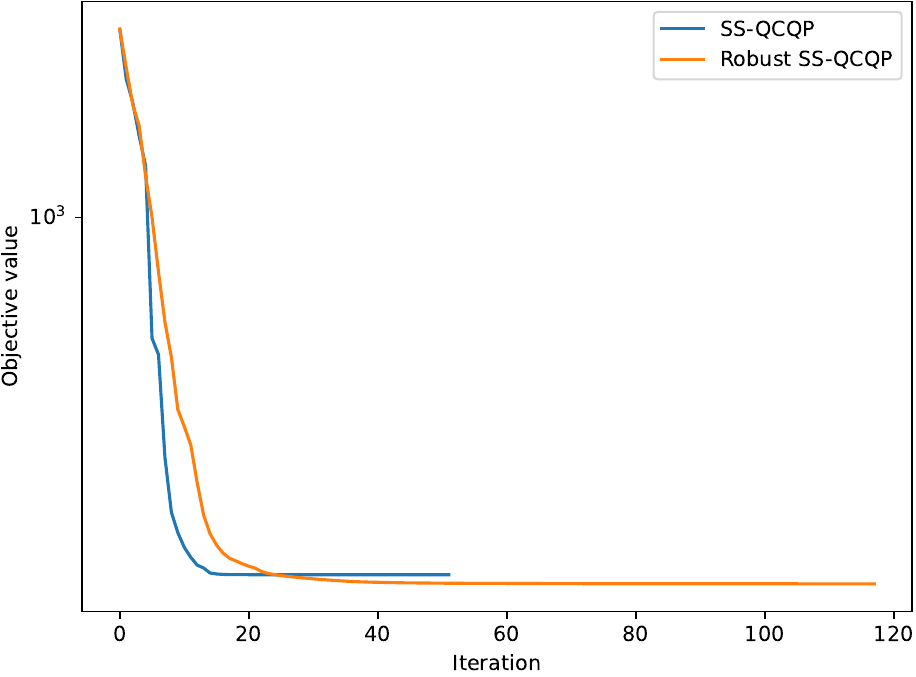}
    \includegraphics[width=0.49\linewidth]{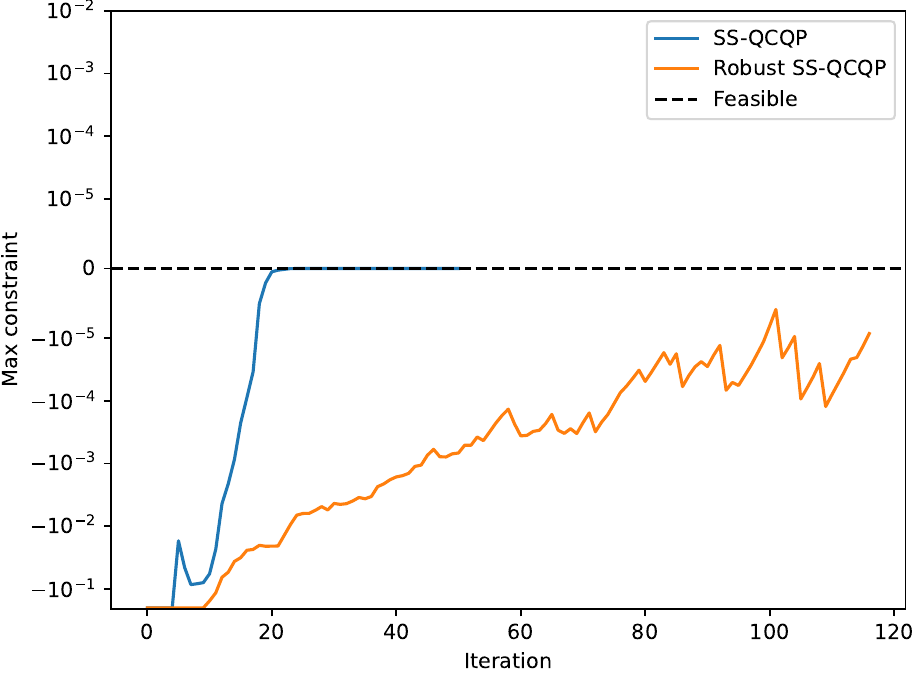}
    \caption{The top-left figure shows the navigation performance of our method with $n_v = 3$. The top-right figure compares our method with SS-QCQP under a noisy setup. The bottom plots compare the objective value and maximum constraint value across iterations.}
    \label{fig:exp}
\end{figure}

\paragraph{Multi-agent scenario}
We first consider a three-vehicle navigation task (top-left of \Cref{fig:exp}). The initial states are $X_1^s \! = \! [-3,-1,0]^\top,
X_2^s \! = \! [-3,-3,0]^\top,
X_3^s \! = \! [-1,-3,0]^\top$
and the goal states are $X_1^d \! = \! [2,1,0]^\top,
X_2^d \! = \! [2,2,0]^\top,
X_3^d \! = \! [1,2,0]^\top$.
Dotted curves in the figure represent vehicle trajectories.
To model gradient uncertainty,  we perturb the gradient using noise $d_i$ whose direction is sampled uniformly from the unit sphere and whose magnitude is uniformly sampled from $[0,\epsilon]$. Due to the worst-case treatment of gradient perturbations in \eqref{eq:robust_qcqp}, the proposed method maintains feasibility throughout the optimization process despite noisy gradients. However, the final state does not exactly reach the goal due to the noise in the objective gradient.

\paragraph{Single-agent scenario}
We next consider a single-vehicle task (top-right of \Cref{fig:exp}) with $X_1^s = [-1,-3,0]^\top, X_1^d = [1,2,0]^\top$.
Here, the gradient noise has a fixed all-ones direction with magnitude uniformly sampled from $[0,\epsilon]$. 
%
We compare SS-QCQP from \cite{wang2025anytime} with the proposed method. Since SS-QCQP does not account for gradient noise, the perturbations bias the search direction (e.g., inducing a systematic left turn). Moreover, its line search may fail to terminate near the boundary of feasibility, leading to premature failure of the algorithm. In contrast, our proposed method maintains feasibility throughout the optimization process and proceeds to convergence.
The evolution of the objective value and maximum constraint value is shown in the bottom row of \Cref{fig:exp}.

\section{Conclusion}
We developed an anytime-feasible first-order method for constrained optimization under norm-bounded gradient uncertainty. The method computes robust search directions by solving an SOCP and uses a safeguarded backtracking line search to preserve strict feasibility and ensure sufficient objective decrease. We established a uniform positive lower bound on the accepted step sizes, an $O(1/K)$ bound on the average squared search-direction norm, and convergence of the search directions to zero. Numerical experiments on vehicle navigation illustrate the method's ability to maintain feasibility despite gradient perturbations. The guarantees rely on known gradient-error bounds, exact function evaluations, and a strictly feasible initialization. Extending the framework to uncertain function values and unbounded gradient noise is a direction for future work.


\section{Appendix}

\subsection{KKT Conditions of the Original Problem \eqref{eq:optimization_problem}}
\label{app:kkt_original}

Let $x^\star$ be a local minimizer of~(1) at which MFCQ
holds. Then there exist multipliers
$\lambda^\star\in\mathbb{R}_+^m$ satisfying
\begin{subequations}\label{eq:kkt_original}
\begin{align} \label{eq:kkt_original_stationarity}
    \nabla f_0(x^\star)
    +\sum_{i=1}^m\lambda_i^\star\nabla f_i(x^\star)
    =0,  \\
    f_i(x^\star)\le0, \quad  \lambda_i^\star\ge0,\quad  \lambda_i^\star f_i(x^\star)=0,
    &&i\in[m],
\end{align}

\end{subequations}
A feasible point satisfying these conditions is called
a KKT point of~\eqref{eq:optimization_problem}.


\subsection{KKT conditions of \eqref{eq:robust_qcqp}}
Fix a strictly feasible point $x$, and let $u(x)$ be the unique minimizer of \eqref{eq:robust_qcqp}. By Lemma~\ref{lem:strict_feas}, the subproblem satisfies Slater's condition.
Since it is convex, its KKT conditions are necessary
and sufficient for optimality. Thus, there exist multipliers $\lambda_i(x) \ge 0$, $i \in [m]$, such that
\begin{align}\label{eq:kkt_socp_full}
&0 \!\in \! u(x) \!+\! \magenta{\widehat{ \nabla f}}_0(x) \!+\! \epsilon_0 \partial \|u(x)\| \!+ \!\sum_{i=1}^m \lambda_i(x) \partial_u h_i(u(x)), \notag \\
&h_i(u(x)) \le 0, \quad 
\lambda_i(x) h_i(u(x)) = 0 \quad i \in [m],
\end{align}
\magenta{where $h_i(u) = \widehat{\nabla f}_i(x)^\top u \! + \! \epsilon_i \|u\| \! + \! \alpha f_i(x) \! + \! w_i \|u\|^2$}.
Here, $\partial\|u\|$ denotes the subdifferential
of the Euclidean norm at $u$ \magenta{and $\partial h_i(u) = \widehat{\nabla f}_i(x) + \epsilon_i \partial \|u\| +2w_i u$}.
Since $\|u\|$ is not differentiable at $u = 0$, the stationarity condition must be characterized separately for the cases $u(x) = 0$ and $u(x) \neq 0$. If $ u(x) \neq 0$, then the stationarity condition becomes
\begin{align}\label{eq:kkt_socp_unonzero}
    u(x) &+ \magenta{\widehat{ \nabla f}}_0(x)
    + \epsilon_0 \frac{u(x)}{\|u(x)\|}  \\
    &+ \sum_{i=1}^m \lambda_i(x) \left( \widehat{\nabla f}_i(x) 
    + \epsilon_i \frac{u(x)}{\|u(x)\|} 
    + 2 w_i u(x) \right) = 0.\notag
\end{align}
If $ u(x) = 0$, there exist vectors $v_0,\ldots,v_m$
with $\|v_i\|\le1$ such that
\begin{align}\label{eq:kkt_socp_zero}
\magenta{\widehat{ \nabla f}}_0(x) + \epsilon_0 v_0 + \sum_{i=1}^m \lambda_i(x) (\widehat{\nabla f}_i(x) + \epsilon_iv_i) = 0.
\end{align}

\subsection{Proof of Theorem \ref{thm:descent}} \label{proof:descent}
\begin{proof}
If \(u(x)=0\), the claim is immediate. Consider now the case \(u(x)\neq 0\). Since \(x\) is strictly feasible for \eqref{eq:optimization_problem}, Lemma~\ref{lem:strict_feas} implies that the robust subproblem \eqref{eq:robust_qcqp} satisfies Slater's condition. Moreover, \eqref{eq:robust_qcqp} is convex. Hence, the KKT conditions are necessary and sufficient for optimality.

Let $\lambda_i(x)$ be the associated KKT multipliers. From the stationarity condition in \eqref{eq:kkt_socp_unonzero}, we have
\begin{align}
\magenta{\widehat{ \nabla f}}_0(x) \!
= &\! -u(x) \! - \! \epsilon_0 \frac{u(x)}{\|u(x)\|}
\! \notag \\
 & - \! \sum_{i=1}^m \lambda_i(x) \left(  \widehat{\nabla f}_i(x)  \! + \! \epsilon_i \frac{u(x)}{\|u(x)\|}  \! + \!  2 w_i u(x) \right). \notag
\end{align}
Taking the inner product with $u(x)$ yields
\begin{align}
    &\magenta{\widehat{ \nabla f}}_0(x)^\top u(x) = -\|u(x)\|^2 \! - \! \epsilon_0 \|u(x)\| \notag\\
        & - \! \sum_{i=1}^m \lambda_i(x) \left(
        \widehat{\nabla f}_i(x)^\top u(x)
        + \epsilon_i \|u(x)\|
        + 2 w_i \|u(x)\|^2
        \right) \notag \\
    &= -\|u(x)\|^2 \! - \! \epsilon_0 \|u(x)\| 
    \! + \! \sum_{i=1}^m \lambda_i(x) (\alpha f_i(x) \! - \! w_i \|u(x)\|^2),\notag
\end{align}
where we used complementary slackness \eqref{eq:kkt_socp_full}.
Since $f_i(x) \le 0$ and $\lambda_i \ge 0$, we have
\begin{align}\label{eq:descent_tilde}
\magenta{\widehat{ \nabla f}}_0(x)^\top u(x) + \epsilon_0 \|u(x)\| \le -\|u(x)\|^2.
\end{align}
Finally, since $\|d_i\| \leq \epsilon_i$, 
\begin{align}
    \nabla f_0(x)^\top u(x) 
    &\leq  \sup_{\|d_0\|\leq \epsilon_0} (\magenta{\widehat{ \nabla f}}_0(x) + d_0)^\top u(x) \notag\\
    &\leq \magenta{\widehat{ \nabla f}}_0(x)^\top u(x) + \epsilon_0 \|u(x)\|, \notag
\end{align}
which coupled with \eqref{eq:descent_tilde}, concludes the proof.
\end{proof}

\subsection{Proof of \Cref{thm:e-kkt}}\label{proof:e-kkt}
\begin{proof}
Since \(u(x)=0\), the SOCP constraints reduce to
\(
\alpha f_i(x)\le 0,\quad \forall i\in[m].
\)
Because \(x\) is strictly feasible, we have
\(
\alpha f_i(x)<0,\quad \forall i\in[m].
\)
Hence all SOCP constraints are strictly inactive at \(u(x)=0\). By complementary slackness of \eqref{eq:robust_qcqp}, every associated KKT multiplier must therefore vanish:
\(
\lambda_i(x)=0,\  \forall i\in[m]
\).
\magenta{
Using the uncertainty model \eqref{eq:noisy_grad} and that $\partial \|u\|\big|_{u=0}=\{q\in\mathbb R^n:\|q\|\le 1\}$, the stationarity condition of \eqref{eq:robust_qcqp} reduce to 
\begin{align*}
    \nabla f_0(x) \!+\! d_0 \!+\! \epsilon_0 q \!+\! \sum_{i=1}^m \lambda_i(x) 
    \big[ 
        \nabla f_i(x) \!+\! d_i \!+\! \epsilon_i q
    \big] \!=\! 0,
\end{align*}
where $\|q\| \leq 1, \|d_i\| \leq \epsilon_i$. Since $\lambda_i(x)=0 ~\forall i \in [m]$, we have
\begin{align*}
    \left\|\nabla f_0(x)+\sum_{i=1}^m \lambda_i(x)\nabla f_i(x)\right\| = \|d_0 + \epsilon_0 q\| \leq 2\epsilon_0.
\end{align*}
Finally, since \(x\in\mathcal F\), primal feasibility holds. Moreover, because \(\lambda_i(x)=0\),
\(
\lambda_i(x)f_i(x)=0
\) for all \(i\). Hence \(x\) is a \(2\epsilon_0\)-KKT point of \eqref{eq:optimization_problem} with dual variable $\lambda_i = \lambda_i(x)$.
}
%
\end{proof}


\subsection{Proof of Proposition \ref{prop:stepsize}}\label{proof:stepsize}
\begin{proof}
We show that both conditions \eqref{eq:armijo}--\eqref{eq:feas_ls} hold for all sufficiently small $t > 0$, and then extract an explicit bound. Following \Cref{lem:strict_feas}, since $x^k$ is strictly feasible for \eqref{eq:optimization_problem}, KKT conditions hold for \eqref{eq:robust_qcqp}. 

\smallskip
\noindent
\textit{Descent condition:}
By Lipschitz continuity of $\nabla f_0$, we have
\begin{align}
    f_0(x^k +& t u(x^k))  \notag\\
    \le& f_0(x^k) \! + \! t \nabla f_0(x^k)^\top u(x^k) \! + \!  \frac{L_{0}}{2} t^2 \|u(x^k)\|^2 \notag \\
    \le& f_0(x^k) \! - \! (t - \frac{t^2L_0}{2})\|u(x^k)\|^2,\notag
\end{align}
where we used \Cref{thm:descent}.
Thus, \eqref{eq:armijo} holds if
\begin{align}
     & f_0 (x^k) - (t - \frac{t^2L_0}{2})\|u(x^k)\|^2  
    \le f_0 (x^k) -\gamma_1 t \|u(x^k)\|^2, \notag
\end{align}
which simplifies to $t \le \frac{2(1-\gamma_1)}{L_{0}}.$

\smallskip
\noindent
\textit{Feasibility condition:}
By Lipschitz continuity of $\nabla f_i$, 
\begin{align}
f_i&(x^k + t u(x^k)) \notag \\
\le& f_i(x^k) + t \nabla f_i(x^k)^\top u(x^k) + \frac{L_{i}}{2} t^2 \|u(x^k)\|^2 \notag \\
\le& f_i(x^k) \! - \! t\alpha f_i(x^k) - t w_i\|u(x^k)\|^2 + \frac{L_{i}}{2} t^2 \|u(x^k)\|^2 \notag \\
=& (1- t\alpha) f_i(x^k) + \left(\frac{L_{i}}{2} t^2 - t w_i\right) \|u(x^k)\|^2, \notag
\end{align}
where we used the feasibility of $u(x^k)$, i.e.,
\begin{align}
    \nabla f_i(x^k)^\top u(x^k)
     &\le \! \nabla \tilde{f}_i(x^k)^\top u(x^k) \! +  \! \epsilon_i \|u(x^k)\|  \notag\\
     &\leq \! -\alpha f_i(x^k) \! - \! w_i\|u(x^k)\|^2. \notag
\end{align}
Thus, \eqref{eq:feas_ls} holds if
\begin{align}
(1 \!-\! \alpha t) f_i(x^k) + \left(\frac{L_{i}}{2} t^2 \!-\! t w_i\right) \|u(x^k)\|^2
\le (1 \!-\! \gamma_2) f_i(x^k). \notag
\end{align}
Since $f_i(x^k) < 0$, it suffices that both 
$t \le \frac{\gamma_2}{\alpha}$ and $\frac{L_{i}}{2} t^2 - t w_i \le 0,$
which holds whenever $t \le \min\left\{\frac{\gamma_2}{\alpha}, \frac{2w_i}{L_{i}}\right\}.$
Thus, both conditions are satisfied for
\[
t \le \min\left\{ \frac{2(1-\gamma_1)}{L_0}, \min_i\left(\frac{2w_i}{L_{i}}\right), \frac{\gamma_2}{\alpha} \right\}.
\]
\magenta{Therefore, the resulting step size satisfies $t^k \geq t_{\min}$}.
%
\end{proof}

\subsection{Proof of Theorem \ref{thm:convergence}}
\begin{proof}
We first establish strict feasibility by induction.
The initial point $x^0$ is strictly feasible by assumption.
If $x^k$ is strictly feasible, Proposition~1 guarantees
that the backtracking line search terminates with an
accepted step size $t_k\ge t_{\min}>0$. The feasibility
condition~(10) then gives
\[
    f_i(x^{k+1})
    \le (1-\gamma_2)f_i(x^k)<0,
    \qquad i\in[m],
\]
because $0<\gamma_2<1$.
Thus, every iterate remains strictly feasible.
Applying this inequality recursively proves
\eqref{eq:strict_feasibility_bound}.

Next, the sufficient decrease condition~(9) and the
uniform step-size bound imply
\begin{align}
    f_0(x^{k+1}) \! \le \! f_0(x^k) \! -\! \gamma_1t_k\|u(x^k)\|^2 \! \le \! f_0(x^k) \!- \! \gamma_1t_{\min}\|u(x^k)\|^2 \notag.
\end{align}
Summing from $k=0$ to $K-1$ yields
\[
    \gamma_1t_{\min}
    \sum_{k=0}^{K-1}\|u(x^k)\|^2
    \le f_0(x^0)-f_0(x^K).
\]
Since $x^K$ is feasible, $f_0(x^K)\ge f^\star$.
Dividing by $\gamma_1t_{\min}K$ therefore establishes
\eqref{eq:ergodic_direction_bound}. The same inequality shows that 
\[
    \sum_{k=0}^{K-1}\|u(x^k)\|^2
    \le
    \frac{f_0(x^0)-f^\star}{\gamma_1t_{\min}}
    \qquad \forall K\ge1.
\]
These partial sums are nondecreasing and uniformly bounded.
As $K\!\to\!\infty$, we conclude that the series converges,
which implies $\|u(x^k)\|^2\!\to\!0$.
\end{proof}

\bibliographystyle{ieeetr}
\bibliography{refs}

\end{document}